\documentclass[fleq]{amsart}
\usepackage{amssymb,amsmath,latexsym,amsbsy,color,enumerate}
\numberwithin{equation}{section}

\newtheorem{theorem}{Theorem}
\newtheorem{definition}[theorem]{Definition}
\newtheorem{lemma}{Lemma}
\newtheorem{remark}{Remark}
\newtheorem{proposition}{Proposition}

\begin{document}
 \title{Quasi-potentials and regularization of currents, and applications}
 \author{Tuyen Trung Truong}
    \address{Indiana University Bloomington IN 47405}
 \email{truongt@indiana.edu}
\thanks{}
    \date{\today}
    \begin{abstract}
Let $Y$ be a compact K\"ahler manifold. We show that the weak regularization $K_n$ of Dinh and Sibony for the diagonal $\Delta _Y$ (see Section 2 for more detail) is compatible with wedge product in the following sense: 

If $T$ is a positive $dd^c$-closed $(p,p)$ current and $\theta$ is a smooth $(q,q)$ form then there is a sequence of positive $dd^c$-closed $(p+q,p+q)$ currents $S_n$ whose masses converge to $0$ so that $-S_n\leq K_n(T\wedge \theta )-K_n(T)\wedge \theta \leq S_n$ for all $n$.

We also prove a result concerning the quasi-potentials of positive closed currents. We give two applications of these results. First, we prove a corresponding compatibility with wedge product for the pullback operator defined in our previous paper. Second, we define an intersection product for positive $dd^c$-closed currents. This intersection is symmetric and has a local nature.   
\end{abstract}
\maketitle
\section{Introduction}       
\label{SectionIntroduction} 

This paper continues the work of our previous paper \cite{truong} on pullback of currents. Here we prove a compatible property on pullingback of a current of the form $T\wedge \theta$, where $T$ is a pseudo-$dd^c$-plurisubharmonic $(p,p)$ current and $\theta$ is a smooth $(q,q)$ form. We will also define an intersection product for positive $dd^c$-closed currents. This intersection is symmetric and has a local nature.

We will need the following two technical results. The first concerns the quasi-potential of a positive closed $(p,p)$ current $T$ on a compact K\"ahler manifold $Y$. It is known that (see Dinh and Sibony\cite{dinh-sibony5}, Bost, Gillet and Soule\cite{bost-gillet-soule}) there is a $DSH$ (p-1,p-1) current $S$ and a closed smooth form $\alpha$ so that $T=\alpha +dd^cS$. (The definition of $DSH$ currents, which was given by Dinh and Sibony \cite{dinh-sibony1}, will be recalled in Section 3). Here $S$ is a difference of two negative currents. When $p=1$ or $Y$ is a projective space, then we can choose $S$ to be negative. However in general we can not choose $S$ to be negative (see \cite{bost-gillet-soule}). The following weaker conclusion is sufficient for the purpose of this paper

\begin{lemma}
Let $T$ be a positive closed $(p,p)$ current on a compact K\"ahler manifold $Y$. Then there is a closed smooth $(p,p)$ form $\alpha$ and a negative $DSH$ $(p-1,p-1)$ current $S$ so that 
\begin{eqnarray*}
T\leq \alpha +dd^cS.
\end{eqnarray*}
Moreover, there is a constant $C>0$ independent of $T$ so that $||\alpha ||_{L^{\infty}}\leq C||T||$ and $||S||\leq C||T||$. If $T$ is strongly positive then we can choose $S$ to be strongly negative.
\label{LemmaQuasiPotential}\end{lemma}   
Here $||.||_{L^{\infty}}$ is the maximum norm of a continuous form and $||.||$ is the mass of a positive or negative current.

The second technical result concerns regularization of currents. If $Y$ is a compact K\"ahler manifold, we let $\pi _1,\pi _2:Y\times Y\rightarrow Y$ the projections. If $K$ is a current on $Y\times Y$ and $T$ a current on $Y$, we define $K(T)=(\pi _1)_*(K\wedge \pi _2^*(T))$, whenever the wedge product $K\wedge \pi _2^*(T)$ makes sense. 
\begin{lemma} Let $Y$ be a compact K\"ahler manifold. Let $K_n$ be a weak regularization of the diagonal $\Delta _Y$ defined in \cite{dinh-sibony1} (see Section 2 for more detail). Let $T$ be a $DSH$ $(p,p)$ current and let $\theta$ be a continuous $(q,q)$ form on $Y$. Assume that there is a positive $dd^c$-closed current $R$ so that $-R\leq T\leq R$. Then there are positive $dd^c$-closed $(p+q,p+q)$ currents $R_n$ so that $\lim _{n\rightarrow\infty}||R_n||=0$ and 
\begin{eqnarray*}
-R_n\leq K_n(T\wedge \theta )-K_n(T)\wedge \theta\leq R_n, 
\end{eqnarray*}
for all $n$. 

If $R$ is strongly positive or closed then we can choose $R_n$ to be so. 
\label{TheoremRegularizationCompactibleWithWedgproduct}\end{lemma}  

Now we present some consequences of Lemmas \ref{LemmaQuasiPotential} and \ref{TheoremRegularizationCompactibleWithWedgproduct}. We discuss first the application to pullback of currents. Let $X$ and $Y$ be compact K\"ahler manifolds and let $f:X\rightarrow Y$ be a dominant meromorphic map. In \cite{truong} we defined a pullback operator $f^{\sharp}$ for currents on $Y$ as follows. Let $s\geq 0$ be an integer. Then a good approximation scheme by $C^s$ forms is an approximation for all $DSH$ currents by $C^s$ forms and  satisfies a list of requirements (See Definition \ref{DefinitionGoodApproximation} in Section 3. Note that the definition of good approximation schemes here is stronger than that in \cite{truong} because here we require it to satisfy in addition the conclusions of Lemma \ref{TheoremRegularizationCompactibleWithWedgproduct}). Because $Y$ is compact, if $T$ is a current on $Y$ then it is of a finite order $s_0$. We say that $f^{\sharp}(T)=S$ is well-defined if there is a number $s\geq s_0$ such that for any good approximation by $C^{s+2}$ forms $\mathcal{K}_n$ then for any smooth form $\alpha$ on $X$ we have 
\begin{eqnarray*}
\lim _{n\rightarrow\infty}\int _{Y}T\wedge \mathcal{K}_n(f_*(\alpha ))=\int _XS\wedge \alpha . 
\end{eqnarray*}

Let $\Gamma _f$ be the graph of $f$. Let $\pi _X,\pi _Y:X\times Y\rightarrow X,Y$ be the projections. A current $\tau$ is called pseudo-$dd^c$-plurisubharmonic if there is a smooth form $\gamma $ so that $dd^c\tau\geq -\gamma $. We have the following result
\begin{theorem}

Let $T$ be a $DSH$ $(p,p)$ current and let $\theta$ be a smooth $(q,q)$ form on $Y$. Assume that there is a positive pseudo-$dd^c$-plurisubharmonic current $\tau$ so that $-\tau\leq T\leq \tau$.

a) If $f$ is holomorphic and $f^{\sharp}(T)$ is well-defined, then $f^{\sharp}(T\wedge \theta )$ is well-defined. Moreover, $f^{\sharp}(T\wedge \theta )=f^{\sharp}(T)\wedge f^*(\theta )$. 

b) More general, assume that there is a number $s\geq 0$ and a $(p,p)$ current $(\pi _Y|\Gamma _f)^{\sharp}(T)$ on $X\times Y$  such that for any good approximation by $C^{s+2}$ forms $\mathcal{K}_n$ then 
\begin{eqnarray*} 
\lim _{n\rightarrow\infty}\pi _Y^*(\mathcal{K}_n(T))\wedge [\Gamma _f]=(\pi _Y|\Gamma _f)^{\sharp}(T).
\end{eqnarray*}
Then $f^{\sharp}(T\wedge \theta )$ is well-defined, and moreover $f^{\sharp}(T\wedge \theta )=(\pi _X)_*((\pi _Y|\Gamma _f)^{\sharp}(T)\wedge \pi _Y^*(\theta ))$.
\label{TheoremPullbackCompatibleWithWedgeProduct}\end{theorem}
Roughly speaking, the result b) of Theorem \ref{TheoremPullbackCompatibleWithWedgeProduct} says that under some natural conditions if we can pullback $T$ then we can do it locally. To illustrate the use of Theorem \ref{TheoremPullbackCompatibleWithWedgeProduct} we will show in Proposition \ref{PropositionPullbackOfFormsWithBoundedCoefficients} in Section 3 the following result: If $T$ is a $(p,p)$ (non-smooth) form whose coefficients are bounded by a quasi-PSH function then $T$ can be pulled back by meromorphic maps. Moreover the resulting current is the same as that defined by Dinh and Sibony (see Proposition 4.2 in \cite{dinh-sibony2}) 

Some special cases of Theorem \ref{TheoremPullbackCompatibleWithWedgeProduct} and Proposition \ref{PropositionPullbackOfFormsWithBoundedCoefficients} have been considered in the literature. Diller \cite{diller} defined for a rational selfmap of $\mathbb{P}^2$ the pullback of a current of the form $\psi T$ where $\psi$ is a smooth function and $T$ is a positive closed $(1,1)$ current. Russakovskii and Shiffman  \cite{russakovskii-shiffman} defined pullback by a holomorphic map for currents of the forms $\varphi \Phi$ (where $\varphi$ is a quasi-plurisubharmonic function and $\Phi$ is a smooth form) and $[D]\wedge \Phi$ (where $D$ is a divisor and $\Phi$ is a smooth form). 

We can also apply Theorem \ref{TheoremPullbackCompatibleWithWedgeProduct} to other situations. Let $f:X\rightarrow Y$ be a dominant meromorphic map between compact K\"ahler manifolds. In \cite{truong}, we showed that if $T$ is a positive $dd^c$-closed $(1,1)$ current then $f^{\sharp}(T)$ is well-defined (the resulting current coincides with the definitions given by Alessandrini and Bassanelli \cite{alessandrini-bassanelli2} and Dinh and Sibony \cite{dinh-sibony2}), and therefore $f^{\sharp}(T\wedge \theta)$ is well-defined for any smooth $(q,q)$ form $\theta$. Likewise, if $V$ is an irreducible analytic variety of codimension $p$ so that $\pi _Y^{-1}(V)\cap \Gamma _f$ has codimension $\geq p$ then $(\pi _Y|\Gamma _f)^{\sharp}[V]$ is well-defined, and therefore $f^{\sharp}([V]\wedge \theta )$ is well-defined for smooth $(q,q)$ forms $\theta$.

Using super-potential theory, Dinh and Sibony \cite{dinh-sibony4} defined a satisfying intersection theory for positive closed currents on a projective space. In the below we give a definition for intersection product of currents on a general compact K\"ahler manifold and discuss some of its properties.
\begin{definition}
Let $Y$ be a compact K\"ahler manifold. Let $T_1$ be a $DSH$ current and let $T_2$ be any current on $Y$. Let $s_0$ be the order of $T_2$. We say that $T_1\wedge T_2$ is well-defined if there is $s\geq s_0$ and a current $S$ so that for any good approximation scheme by $C^{s+2}$ forms $\mathcal{K}_n$ then $\lim _{n\rightarrow\infty}\mathcal{K}_n(T_1)\wedge T_2=S$. Then we write $T_1\wedge T_2=S$. 
\label{DefinitionIntersectionCurrents}\end{definition}  

This definition has the following properties
\begin{theorem}
Let $T_1$ and $T_2$ be positive $dd^c$-closed currents. Assume that $T_1\wedge T_2$ is well-defined. Let $\theta$ be a smooth $(q,q)$ form. 

a) $\theta \wedge T_2$ and $T_2\wedge \theta$ are well-defined and are the same as the usual definition. 

b) $T_2\wedge T_1$ is also well-defined. Moreover, $T_1\wedge T_2=T_2\wedge T_1$.

c) $T_1\wedge (\theta \wedge T_2)$ is also well-defined. Moreover $T_1\wedge (\theta \wedge T_2)=(T_1\wedge T_2)\wedge \theta$.
\label{TheoremSymmertyOfIntersectionCurrents}\end{theorem}
Theorem \ref{TheoremSymmertyOfIntersectionCurrents} b) means that the intersection is symmetric, and Theorem \ref{TheoremSymmertyOfIntersectionCurrents} c) means that the intersection can be computed locally. For intersection of varieties we have the expected result
\begin{lemma}
Let $V_1$ and $V_2$ be irreducible subvarieties of codimensions $p$ and $q$ of $Y$. Assume that any component of $V_1\cap V_2$ has codimension $p+q$. Then $[V_1]\wedge [V_2]$ is well-defined. Here $[V_1]$ and $[V_2]$ are the currents of integration on $V_1$ and $V_2$.
\label{LemmaIntersectionOfVarieties}\end{lemma}
The rest of this paper is organized as follows. In Section 2 we recall the construction of weak regularization for the diagonal and prove Lemmas \ref{LemmaQuasiPotential} and \ref{TheoremRegularizationCompactibleWithWedgproduct}. In Section 3 we prove the other results.  

\section{Proofs of Lemmas \ref{LemmaQuasiPotential} and \ref{TheoremRegularizationCompactibleWithWedgproduct}}

Let $Y$ be a compact Kahler manifold of dimension $k$. Let $\pi _1,\pi _2:Y\times Y\rightarrow Y$ be the two projections, and let $\Delta _Y\subset Y\times Y$ be the diagonal. Let $\omega _Y$ be a K\"ahler $(1,1)$ form on $Y$.

For any $p$, we define $DSH^{p}(Y)$ (see \cite{dinh-sibony1}) to be the space of $(p,p)$ currents $T=T_1-T_2$, where $T_i$ are positive
currents, such that $dd^cT_i=\Omega _i^+-\Omega _i^-$ with $\Omega _i^{\pm}$ positive closed. Observe that $||\Omega _i^{+}||=||\Omega _i^-||$ since they
are cohomologous to each other because $dd^c(T_i)$ is an exact current. Define the $DSH$-norm of $T$ as
\begin{eqnarray*}
||T||_{DSH}:=\min \{||T_1||+||T_2||+||\Omega _1^{+}||+||\Omega _2^{+}||,~T_i,~\Omega _i,~\mbox{as above}\}.
\end{eqnarray*}
Using compactness of positive currents, it can be seen that we can find $T_i,~\Omega _i^{\pm}$ which realize $||T||_{DSH}$, hence the minimum on the RHS
of the definition of $DSH$ norm. We say that $T_n\rightharpoonup T$ in $DSH^p(Y)$ if $T_n$ weakly converges to $T$ and $||T_n||_{DSH}$ is bounded.

Recall that a function $\varphi$ is quasi-PSH if it is upper semi-continuous, belongs to $L^1$, and $dd^c(\varphi )=T-\theta$, where $T$ is a positive closed $(1,1)$ current and $\theta$ is a closed smooth $(1,1)$ form. We also call $\varphi$ a $\theta$-plurisubharmonic function. 
\begin{remark}
The following consideration from \cite{bost-gillet-soule} and \cite{dinh-sibony5} is used in both proof of Lemam \ref{LemmaQuasiPotential} and the construction of the kernels $K_n$ in Lemma \ref{TheoremRegularizationCompactibleWithWedgproduct}. Let $k=$ dimension of $Y$. Let $\pi :\widetilde{Y\times Y}\rightarrow Y\times Y$ be the blowup of $Y\times Y$ at $\Delta _Y$. Let $\widetilde{\Delta} _Y=\pi ^{-1}(\Delta _Y)$ be the exceptional divisor. Then there is a closed smooth $(1,1)$ form $\gamma$ and a negative quasi-plurisubharmonic function $\varphi$ so that $dd^c\varphi =[\widetilde{\Delta} _Y]-\gamma $. We choose a strictly positive closed smooth $(k-1,k-1)$ form $\eta $ so that $\pi _*([\widetilde{\Delta} _Y]\wedge \eta )=[\Delta _Y]$. 
\label{Remark2}\end{remark}
Next we give the proof of Lemma \ref{LemmaQuasiPotential}.
\begin{proof} (Of Lemma \ref{LemmaQuasiPotential})
Notations are as in Remark \ref{Remark2}. Define $H=\pi _*(\varphi \eta )$. Then $H$ is a negative $(k-1,k-1)$ current on $Y\times Y$.  

We write $\gamma =\gamma ^+-\gamma ^-$ for strictly positive closed smooth $(1,1)$ forms $\gamma ^{\pm}$. If we define $\Phi ^{\pm}=\pi _*(\gamma ^{\pm}\wedge \eta )$ then $\Phi ^{\pm}$ are positive closed $(k,k)$ currents with $L^1$ coefficients. In fact (see \cite{dinh-sibony1}) $\Phi ^{\pm}$ are smooth away from the diagonal $\Delta _Y$, and the singularities of $\Phi ^{\pm}(y_1,y_2)$ and their derivatives are bounded by $|y_1-y_2| ^{-(2k-2)}$ and $|y_1-y_2|^{-(2k-1)}$. Moreover 
\begin{eqnarray*} 
dd^cH=\pi _*(dd^c\varphi \wedge \eta )=\pi _*([\widetilde{\Delta} _Y ]\wedge \eta -(\gamma ^+-\gamma ^-)\wedge \eta )=[\Delta _Y]-(\Phi ^+-\Phi ^-).
\end{eqnarray*} 
Consider $S_1=(\pi _1)_*(H\wedge \pi _2^*(T))$ and $R_1^{\pm}=(\pi _1)_*(\Phi ^{\pm}\wedge T)$. Then $S_1$ is a negative current, and $R_1^{\pm}$ are positive closed currents. Moreover 
\begin{eqnarray*}
dd^cS_1=(\pi _1)_*(dd^cH\wedge \pi _2^*(T))=T-R_1^++R_1^-.
\end{eqnarray*}
Therefore $T\leq R_1^++dd^cS_1$. Moreover $R_1^+$ is a current with $L^1$ coefficients, and there is a constant $C_1>0$ independent of $T$ so that $||S_1||, ||R_1||_{L^1}\leq C_1||T||$ (see e.g. Lemma 2.1 in \cite{dinh-sibony1}). 

If we apply this process for $R_1^+$ instead of $T$ we find a positive closed current $R_2^+$ with coefficients in $L^{1+1/(2k+2)}$ and a negative current $S_2$ so that $R_1^+\leq R_2^++dd^cS_2$. Moreover 
\begin{eqnarray*}
||R_2^+||_{L^{1+1/(2k+2)}},||S_2||\leq C_2||R_1^{+}||_{L^1}\leq C_1C_2||T||
\end{eqnarray*}
for some constant $C_2>0$ independent of $T$. After iterating this process a finite number of times we find a continuous form $R$ and a negative current $S$ so that $T\leq R+dd^cS$. Moreover, $||R||_{L^{\infty}},||S||\leq C||T||$ for some constant $C>0$ independent of $T$. Since we can bound $R$ by $\omega _Y^p$ upto a multiple constant of size $||R||_{L^{\infty}}$, we are done.   
 \end{proof}
Next we recall the construction of the kernels $K_n$ from Section 3 in \cite{dinh-sibony1}. Notations are as in Remark \ref{Remark2}. Observe that $\varphi $ is smooth out of $[\widetilde{\Delta _Y}]$, and $\varphi ^{-1}(-\infty )=\widetilde{\Delta _Y}$. Let $\chi :\mathbb{R}\cup \{-\infty\} \rightarrow \mathbb{R}$ be a smooth increasing convex function such that $\chi (x)=0$ on $[-\infty ,-1]$, $\chi (x)=x$ on $[1, +\infty ]$, and $0\leq \chi '\leq 1$.  Define $\chi _n(x)=\chi (x+n)-n$, and $\varphi _n=\chi _n\circ \varphi $. The functions $\varphi _n$ are smooth decreasing to $\varphi$, and $dd^c \varphi _n\geq -\Theta $ for every $n$, where $\Theta$ is a strictly positive closed smooth $(1,1)$ form so that $\Theta -\gamma$ is strictly positive. Then we define $\Theta _n^+=dd^c \varphi _n+\Theta $ and $\Theta _n^-=\Theta ^-=\Theta -\gamma$. Finally $K_n^{\pm}=\pi _*(\Theta _n^{\pm}\wedge \eta )$, and $K_n=K_n^{+}-K_n^-$.     

\begin{proof} (Of Lemma \ref{TheoremRegularizationCompactibleWithWedgproduct})

Let us define $H_n=K_n(T\wedge \theta )-K_n(T)\wedge \theta$. Since $T$ and $\theta$ may not be either positive or $dd^c$-closed, a priori $H_n$ is neither. However, we will show that there are positive $dd^c$-closed currents $R_n$ such that $\lim _{n\rightarrow\infty}||R_n||=0$ and $-R_n\leq H_n\leq R_n$. 

By definition we have
\begin{eqnarray*}
H_n(y)=\int _{z\in Y}K_n(y,z)\wedge (\theta (z)-\theta (y))\wedge T(z).
\end{eqnarray*}
Fix a number $\delta >0$. Then by the construction of $K_n$, there is an integer $n_{\delta}$ so that if $n\geq n_{\delta}$ and $|y-z|\geq \delta$ then $K_n(y,z)=0$. Thus
\begin{eqnarray*}
H_n(y)=\int _{z\in Y,~|z-y|<\delta}(K_n^+(y,z)-K_n^-(y,z))\wedge (\theta (z)-\theta (y))\wedge T(z).
\end{eqnarray*}
We define $h (\delta )=\max _{y,z\in Y: ~|y-z|\leq\delta}|\theta (y)-\theta (z)|$. Because $\theta$ is a continuous form, we have $\lim _{\delta\rightarrow 0}h(\delta )=0$. Moreover, since $Y\times Y$ is compact, there is a constant $C>0$ independent of $\theta$ and $\delta$ so that 
\begin{eqnarray*}
-h(\delta )C(\omega _Y(y)+\omega _Y(z))^q\leq \theta (z)-\theta (y)\leq h(\delta )C(\omega _Y(y)+\omega _Y(z))^q
\end{eqnarray*}
for all $\delta \leq 1$ and for all $|y-z|\leq \delta$. Since $K_n^{\pm}(y,z)$ are strongly positive closed and $-R\leq T\leq R$, it follows that 

\begin{eqnarray*}
H_n(y)&=&\int _{z\in Y,~|z-y|<\delta}(K_n^+(y,z)-K_n^-(y,z))\wedge (\theta (z)-\theta (y))\wedge T(z)\\
&\leq&h(\delta)C\int _{z\in Y,~|z-y|<\delta}(K_n^+(y,z)+K_n^-(y,z))\wedge (\omega _Y(y)+\omega _Y(z))^q\wedge R(z)\\
&\leq&h(\delta)C\int _{z\in Y}(K_n^+(y,z)+K_n^-(y,z))\wedge (\omega _Y(y)+\omega _Y(z))^q\wedge R(z).
\end{eqnarray*}
Thus $H_n(y)\leq R_n(y)$ where 
\begin{eqnarray*}
R_n(y)=h(\delta)C\int _{z\in Y}(K_n^+(y,z)+K_n^-(y,z))\wedge (\omega _Y(y)+\omega _Y(z))^q\wedge R(z),
\end{eqnarray*}
for $n_{\delta }\leq n<n_{\delta /2} $. Similarly we have $H_n(y)\geq -R_n(y)$. It can be checked that $R_n(y)$ is positive $dd^c$-closed. Moreover, there is a constant $C_1>0$ independent of $n$, $\delta$, $R$ and $\theta$ so that 
\begin{equation}
||R_n||\leq h(\delta )C_1||R||, \label{EquationEstimateSn}
\end{equation}
for $n\geq n_{\delta}$. This shows that $||R_n||\rightarrow 0$ as $n\rightarrow \infty$.
\end{proof}
\begin{remark}
By the estimate (\ref{EquationEstimateSn}) and by iterating we obtain the following result: Let $T$, $R$ and $\theta$ be as in Lemma \ref{TheoremRegularizationCompactibleWithWedgproduct}. Then there are positive $dd^c$-closed $(p+q,p+q)$ currents $R_{n_1,n_2,\ldots ,n_l}$ so that
\begin{eqnarray*}
-R_{n_1,n_2,\ldots ,n_l}\leq K_{n_1}\circ K_{n_2}\circ \ldots K_{n_l}(T\wedge \theta )-K_{n_1}\circ K_{n_2}\circ \ldots K_{n_l}(T)\wedge \theta \leq R_{n_1,n_2,\ldots ,n_l},
\end{eqnarray*}
and 
\begin{eqnarray*}
\lim _{n_1,n_2,\ldots ,n_l\rightarrow \infty}||R_{n_1,n_2,\ldots ,n_l}||=0.
\end{eqnarray*}

We give the proof of this claim for example when $l=2$. We will write the $R_n$ in Lemma \ref{TheoremRegularizationCompactibleWithWedgproduct} by $R_n(R)$ to emphasize its dependence on $R$. Writing 
\begin{eqnarray*}
&&K_{n_1}\circ K_{n_2}(T\wedge \theta )-K_{n_1}\circ K_{n_2}(T)\wedge \theta\\
&=&[K_{n_1}(K_{n_2}(T\wedge \theta )-K_{n_2}(T)\wedge \theta )]+[K_{n_1}(K_{n_2}(T)\wedge \theta )-K_{n_1}(K_{n_2}(T))\wedge \theta ],
\end{eqnarray*}
and choosing 
\begin{eqnarray*}
R_{n_1,n_2}=K_{n_1}^+(R_{n_2}(R))+K_{n_1}^-(R_{n_2}(R))+R_{n_1}(K_{n_2}^+(R))+R_{n_1}(K_{n_2}^-(R)),
\end{eqnarray*}
we see that 
\begin{eqnarray*}
-R_{n_1,n_2}\leq K_{n_1}\circ K_{n_2}(T\wedge \theta )-K_{n_1}\circ K_{n_2}(T)\wedge \theta \leq R_{n_1,n_2}.
\end{eqnarray*}
That $R_{n_1,n_2}$ are positive $dd^c$-closed follows from the properties of the kernels $K_n$. It remains to bound the masses of $R_{n_1,n_2}$. By (\ref{EquationEstimateSn}) we have
\begin{eqnarray*}
||R_{n_1,n_2}||&\leq& C_1(||R_{n_2}(R)||+||R_{n_1}(K_{n_2}^+(R))||+||R_{n_2}(K_{n_2}^-(R))||)\\
&\leq&C_2h(\delta )(||R||+||K_{n_2}^+(R)||+||K_{n_2}^-(R)||)\\
&\leq &C_3 h(\delta )||R||,
\end{eqnarray*}
for constants $C_1,C_2,C_3$ and for all $n_1,n_2\geq n_{\delta}$, here $n_{\delta}$ is the constant in the proof of Lemma \ref{TheoremRegularizationCompactibleWithWedgproduct}.
\label{Remark1}\end{remark}

\section{Proofs of the consequences}

We first give the definition of a good approximation scheme by $C^{s}$ forms for $DSH$ currents.
\begin{definition}
Let $Y$ be a compact Kahler manifold. Let $s\geq 0$ be an integer. We define a good approximation scheme by $C^s$ forms for $DSH$ currents on $Y$ to be an assignment that for a $DSH$ current $T$ gives    two sequences $\mathcal{K}_n^{\pm}(T)$ (here $n=1,2,\ldots $) where $\mathcal{K}_n^{\pm}(T)$ are $C^s$ forms of the same bidegrees as $T$, so that $\mathcal{K}_n(T)=\mathcal{K}_n^+(T)-\mathcal{K}_n^-(T)$  weakly converges to $T$, and moreover the following properties are satisfied:

1) Boundedness: The $DSH$ norms of  $\mathcal{K}_n^{\pm}(T)$ are uniformly bounded.

2) Positivity: If $T$ is positive then $\mathcal{K}_n^{\pm}(T)$ are positive, and $||\mathcal{K}_n^{\pm}(T)||$ is uniformly bounded with respect to n.

3) Closedness: If $T$ is positive closed then $\mathcal{K}_n^{\pm}(T)$ are positive closed.  

4) Continuity: If $U\subset Y$ is an open set so that $T|_U$ is a continuous form then $\mathcal{K}_n^{\pm}(T)$ converges locally uniformly on $U$.

5) Additivity: If $T_1$ and $T_2$ are two $DSH^p$ currents, then $\mathcal{K}_n^{\pm}(T_1+T_2)=\mathcal{K}_n^{\pm}(T_1)+\mathcal{K}_n^{\pm}(T_2)$.

6) Commutativity: If $T$ and $S$ are $DSH$ currents with complements bidegrees then 
\begin{eqnarray*}
\lim _{n\rightarrow\infty}[\int _Y\mathcal{K}_n(T)\wedge S-\int _YT\wedge \mathcal{K}_n(S)]=0. 
\end{eqnarray*}

7) Compatibility with the differentials: $dd^c\mathcal{K}_n^{\pm}(T)=\mathcal{K}_n^{\pm}(dd^cT)$.

8) Condition on support: The support of $\mathcal{K}_n(T)$ converges to the support of $T$. By this we mean that if $U$ is an open neighborhood of $supp(T)$, then  there is $n_0$ so that when $n\geq n_0$ then $supp(\mathcal{K}_n(T))$ is contained in $U$. Moreover, the number $n_0$ can be chosen so that it depends only on $supp(T)$ and $U$ but not on the current $T$.

9) Compatibility with wedge product: Let $T$ be a $DSH$ $(p,p)$ current and let $\theta$ be a continuous $(q,q)$ form on $Y$. Assume that there is a positive $dd^c$-closed current $R$ so that $-R\leq T\leq R$. Then there are positive $dd^c$-closed $(p+q,p+q)$ currents $R_n$ so that $\lim _{n\rightarrow\infty}||R_n||=0$ and 
\begin{eqnarray*}
-R_n\leq \mathcal{K}_n(T\wedge \theta )-\mathcal{K}_n(T)\wedge \theta\leq R_n, 
\end{eqnarray*}
for all $n$. 

If $R$ is strongly positive or closed then we can choose $R_n$ to be so.
\label{DefinitionGoodApproximation}\end{definition}

Let $K_n$ be the weak regularization for the diagonal $\Delta _Y$ as in Section 2. Let $l$ be a large integer dependent on $s$, and let $(m_1)_n,\ldots ,(m_l)_n$ be sequences of positive integers satisfying $(m_i)_n=(m_{l+1-i})_n$ and $\lim _{n\rightarrow\infty}(m_i)_n=\infty$ for any $1\leq i\leq l$. In \cite{truong} we showed that if we choose $\mathcal{K}_n=K_{(m_1)_n}\circ K_{(m_2)_n}\circ \ldots K_{(m_l)_n}$ then it satisfies conditions 1)-8). Remark \ref{Remark1} shows that it also satisfies condition 9).

Note that by condition 6), if $T$ is a $DSH$ current then $f^{\sharp}(T)$ is well-defined iff there is a number $s\geq 0$ and a current $S$ so that for any good approximation scheme by $C^{s+2}$ forms $\mathcal{K}_n$ then $\lim _{n\rightarrow\infty}f^*(\mathcal{K}_n(T))=S$.

\begin{proof} (Of Theorem \ref{TheoremPullbackCompatibleWithWedgeProduct})

a) We let $s\geq 0$ be a number so that for any good approximation scheme by $C^{s+2}$ forms $\mathcal{K}_n$ and for any smooth form $\alpha$ on $X$ then 
\begin{eqnarray*}
\int _Xf^{\sharp}(T)\wedge \alpha =\lim _{n\rightarrow\infty}\int _YT\wedge \mathcal{K}_n(f_*(\alpha )).
\end{eqnarray*}

Then for the proof of a) it suffices to show that for any smooth form $\beta$ on $X$ then

\begin{eqnarray*}
\lim _{n\rightarrow\infty}\int _YT\wedge \theta \wedge \mathcal{K}_n(f_*(\beta ))=\int _Xf^{\sharp}(T)\wedge f^*(\theta )\wedge \beta .
\end{eqnarray*}
If we can show 
\begin{equation}
\lim _{n\rightarrow\infty}\int _YT\wedge (\theta \wedge \mathcal{K}_n(f_*(\beta ))-\mathcal{K}_n(\theta \wedge f_*(\beta )))=0
\label{Equation4}
\end{equation}
then we are done, since we have $\theta \wedge f_*(\beta ))=f_*(f^*(\theta )\wedge \beta  )$ because $f$ is holomorphic, and hence
\begin{eqnarray*}
\lim _{n\rightarrow\infty}\int _YT\wedge \mathcal{K}_n(\theta \wedge f_*(\beta ))=\lim _{n\rightarrow\infty}\int _YT\wedge \mathcal{K}_n(f_*(f^*(\theta )\wedge \beta  ))=\int _Yf^{\sharp}(T)\wedge (f^*(\theta )\wedge \beta ).
\end{eqnarray*}
Now we proceed to proving (\ref{Equation4}). For a fixed $n$ we have 
\begin{eqnarray*}
&&\int _YT\wedge (\theta \wedge \mathcal{K}_n(f_*(\beta ))-\mathcal{K}_n(\theta \wedge f_*(\beta )))\\
&=&\lim _{m\rightarrow\infty}\int _Y\mathcal{K}_m(T)\wedge (\theta \wedge \mathcal{K}_n(f_*(\beta ))-\mathcal{K}_n(\theta \wedge f_*(\beta ))).
\end{eqnarray*} 
The advantage of this is that $\mathcal{K}_m(T)$ are continuous forms, hence if we have  bounds of $\theta \wedge \mathcal{K}_n(f_*(\beta ))-\mathcal{K}_n(\theta \wedge f_*(\beta ))$ by currents of order zero we can use them in the integral and then take limit when $m\rightarrow \infty$.

Because $f_*(\beta )$ is bound by a multiple of $f_*(\omega _X^{dim (X)-p-q})$ and the latter is strongly positive closed, by condition 9) of Definition \ref{DefinitionGoodApproximation} there are strongly positive closed currents $R_n$ with $||R_n||\rightarrow 0$ and 
$$-R_n\leq \theta \wedge \mathcal{K}_n(f_*(\beta ))-\mathcal{K}_n(\theta \wedge f_*(\beta ))\leq R_n,$$ for all $n$. Since $-\tau \leq T\leq \tau $, we have $-(\mathcal{K}_m^{+}(\tau )+\mathcal{K}_m^{-}(\tau ))\leq \mathcal{K}_m(T)\leq \mathcal{K}_m^{+}(\tau )+\mathcal{K}_m^{-}(\tau )$. Since $\mathcal{K}_m^{+}(\tau )+\mathcal{K}_m^{-}(\tau )$ are positive $C^2$ forms, from the above estimates we obtain
\begin{eqnarray*}
-\int _Y(\mathcal{K}_m^+(\tau )+\mathcal{K}_m^-(\tau ))\wedge R_n&\leq& \int _Y\mathcal{K}_m(T)\wedge (\theta \wedge \mathcal{K}_n(f_*(\beta ))-\mathcal{K}_n(\theta \wedge f_*(\beta )))\\
&\leq& \int _Y(\mathcal{K}_m^+(\tau )+\mathcal{K}_m^-(\tau ))\wedge R_n.
\end{eqnarray*}
Hence (\ref{Equation4}) follows if we can show that 
\begin{eqnarray*} 
\lim _{n\rightarrow\infty}\lim _{m\rightarrow\infty}\int _Y(\mathcal{K}_m^+(\tau )+\mathcal{K}_m^-(\tau ))\wedge R_n=0.
\end{eqnarray*}
By Lemma \ref{LemmaQuasiPotential}, there are a smooth closed form $\alpha _n$ and a strongly negative current $S_n$ for which $R_n\leq \alpha _n+dd^cS_n$ and $||\alpha _n||_{L^{\infty}}, ||S_n||\rightarrow 0$. Therefore 
\begin{eqnarray*}  
0&\leq& \int _Y(\mathcal{K}_m^+(\tau )+\mathcal{K}_m^-(\tau ))\wedge R_n\\
&\leq& \int _Y(\mathcal{K}_m^+(\tau )+\mathcal{K}_m^-(\tau ))\wedge \alpha _n+\int _Y(\mathcal{K}_m^+(\tau )+\mathcal{K}_m^-(\tau ))\wedge dd^cS_n . 
\end{eqnarray*}
Since the currents $\mathcal{K}_m^{\pm}(\tau )$ are positive whose masses are uniformly bounded, it follows from $||\alpha _n||_{L^{\infty}}\rightarrow 0$ that 
\begin{eqnarray*}
\lim _{n\rightarrow\infty}\lim _{m\rightarrow\infty}\int _Y(\mathcal{K}_m^+(\tau )+\mathcal{K}_m^-(\tau ))\wedge \alpha _n=0.
\end{eqnarray*}
Now we estimate the other term. We have 
\begin{eqnarray*}
\int _Y(\mathcal{K}_m^+(\tau )+\mathcal{K}_m^-(\tau ))\wedge dd^cS_n=\int _Y(\mathcal{K}_m^+(dd^c\tau )+\mathcal{K}_m^-(dd^c\tau ))\wedge S_n.
\end{eqnarray*}
Because $S_n$ is strongly negative and $dd^c\tau \geq -\gamma $, the last integral can be bound from above by
\begin{eqnarray*}
\int _Y(\mathcal{K}_m^+(dd^c\tau )+\mathcal{K}_m^-(dd^c\tau ))\wedge S_n\leq \int _Y(\mathcal{K}_m^+(-\gamma  )+\mathcal{K}_m^-(-\gamma ))\wedge S_n.
\end{eqnarray*}
Since $\gamma $ is smooth, by condition 4) of Definition \ref{DefinitionGoodApproximation} and the fact that $||S_n||\rightarrow 0$, we obtain
\begin{eqnarray*}
\lim _{n\rightarrow\infty}\lim _{m\rightarrow\infty}\int _Y(\mathcal{K}_m^+(-\gamma  )+\mathcal{K}_m^-(-\gamma ))\wedge S_n=0.
\end{eqnarray*}
Thus, whatever the limit of 
\begin{eqnarray*}
\int _Y(\mathcal{K}_m^+(\tau )+\mathcal{K}_m^-(\tau ))\wedge dd^cS_n
\end{eqnarray*} 
is, it is non-positive. The proof of (\ref{Equation4}) and hence of a) is finished.

b) The proof of b) is similar to that of a).
\end{proof}

Now we give an application to pulling back of (non-smooth) forms whose coefficients are bounded by a quasi-PSH function. 
\begin{proposition}
Let $T$ be a $(p,p)$ form whose coefficients are bounded by a quasi-PSH function $\varphi$. Then $f^{\sharp}(T)$ is well-defined. 
\label{PropositionPullbackOfFormsWithBoundedCoefficients}
\end{proposition}
\begin{proof}
By desingularizing the graph $\Gamma _f$ if needed and using Theorem 4 in \cite{truong}, we can assume without loss of generality that $f$ is holomorphic. By subtracting a constant from $\varphi$ if needed, we can assume that $\varphi \leq 0$. By using partition of unity, we reduce the problem to the case where $T=\psi \theta$ where $\psi$ is a function with $0\geq \psi \geq \varphi$ and $\theta$ is a smooth form. By Theorem \ref{TheoremPullbackCompatibleWithWedgeProduct} a), for a proof of Proposition \ref{PropositionPullbackOfFormsWithBoundedCoefficients} it suffices to show that $f^{\sharp}(\psi )$ is well-defined. To this end we will show the existence of a current $S$ so that for any smooth form $\alpha$ and any good approximation scheme by $C^2$ forms $\mathcal{K}_n$ then
\begin{equation}
\lim _{n\rightarrow\infty}\int _Y\psi \wedge \mathcal{K}_n(f_*(\alpha ))=\int _XS\wedge \alpha . 
\label{EquationProof1}\end{equation}
We define linear functionals $S_n$ and $S_n^{\pm}$ on top forms on $X$ by the formulas
\begin{eqnarray*}
<S_n,\alpha >&=&\int _Y\psi \wedge \mathcal{K}_n(f_*(\alpha )),\\
<S_n^{\pm},\alpha >&=&\int _Y\psi \wedge \mathcal{K}_n^{\pm}(f_*(\alpha )).
\end{eqnarray*}  
Then $S_n=S_n^+-S_n^-$, and it can be checked that $S_n^{\pm}$ are negative $(0,0)$ currents, and hence $S_n$ is a current of order $0$. Moreover, if $\alpha$ is a positive smooth measure then 
\begin{eqnarray*} 
0\geq <S_n^{\pm},\alpha >&=&\int _Y\psi \wedge \mathcal{K}_n^{\pm}(f_*(\alpha ))\\
&\geq&\int _Y\varphi \wedge \mathcal{K}_n^{\pm}(f^*(\alpha ))\\
&=&\int _Xf^*(\mathcal{K}_n^{\pm}(\varphi ))\wedge \alpha .
\end{eqnarray*}
Thus $0\geq S_n^{\pm}\geq f^*(\mathcal{K}_n^{\pm}(\varphi ))$ for all $n$.

Let us write $dd^c(\varphi )=T-\theta $ where $T$ is a positive closed $(1,1)$ current, and $\theta$ is a smooth closed $(1,1)$ form. By property 4) of Definition \ref{DefinitionGoodApproximation}, there is a strictly positive closed smooth $(1,1)$ form $\Theta$ so that $\Theta \geq \mathcal{K}_n^{\pm}(\theta )$ for any $n$. Then $f^*(\mathcal{K}_n^{\pm}(\varphi ))$ are negative $C^2$ forms so that 
\begin{eqnarray*}
dd^cf^*(\mathcal{K}_n^{\pm}(\varphi ))&=&f^*(\mathcal{K}_n^{\pm}(dd^c\varphi ))=f^*(\mathcal{K}_n^{\pm}(T-\theta ))\\
&\geq&f^*(\mathcal{K}_n^{\pm}(-\theta ))\geq -f^*(\Theta ) 
\end{eqnarray*}
for any $n$, i.e they are negative $f^*(\Theta )$-plurisubharmonic functions. Moreover the sequence of currents $f^*(\mathcal{K}_n^{\pm}(\varphi ))$ has uniformly bounded mass (see the proof of Theorem 6 in \cite{truong}). Therefore, by the compactness of this class of functions (see Chapter 1 in \cite{demailly}), after passing to a subsequence if needed, we can assume that $f^*(\mathcal{K}_n^{\pm}(\varphi ))$ converges in $L^1$ to negative functions denoted by $f^*(\varphi ^{\pm})$. Let $S^{\pm}$ be any cluster points of $S_n^{\pm}$. Then $0\geq S^{\pm}\geq f^*(\varphi ^{\pm})$, which shows that any cluster point $S=S^+-S^-$ of $S_n$ has no mass on sets of Lebesgue measure zero. Hence to show that $S$ is uniquely defined, it suffices to show that $S$ is uniquely defined outside a proper analytic subset of $Y$. 

Let $E$ be a proper analytic subset of $Y$ so that $f:X-f^{-1}(E)\rightarrow Y-E$ is a holomorphic submersion. If $\alpha$ is a smooth measure whose support is compactly contained in $X-f^{-1}(E)$ then $f_*(\alpha )$ is a smooth measure on $Y$. Hence by condition 4) of Definition \ref{DefinitionGoodApproximation}, $\mathcal{K}_n(f_*(\alpha ))$ uniformly converges to the smooth measure $f_*(\alpha )$. Then it follows from the definition of $S$ that
\begin{eqnarray*}
<S,\alpha >=\int _{Y}\psi \wedge f_*(\alpha ).
\end{eqnarray*}
Hence $S$ is uniquely defined on $X-E$, and thus it is uniquely defined on the whole $X$, as wanted.
\end{proof}

Finally, we consider the intersection of currents.

\begin{proof}(Of Theorem \ref{TheoremSymmertyOfIntersectionCurrents})

Proof of a): Let $\mathcal{K}_n$ be a good approximation scheme by $C^2$ forms. Then $\mathcal{K}_n(\theta )$ uniformly converges to $\theta$, and hence $\mathcal{K}_n(\theta )\wedge T_2$ converges to the usual intersection $\theta \wedge T_2$.

Let $\alpha$ be a smooth form. Then by conditions 9), 6) and 4) of Definition \ref{DefinitionGoodApproximation}, we have
\begin{eqnarray*} 
\lim _{n\rightarrow\infty}\int _Y\mathcal{K}_n(T_2)\wedge \theta \wedge \alpha&=&\lim _{n\rightarrow\infty}\int _Y\mathcal{K}_n(T_2\wedge \theta ) \wedge \alpha\\
&=&\lim _{n\rightarrow\infty}\int _YT_2\wedge \theta \wedge \mathcal{K}_n(\alpha )\\
&=&\int _YT_2\wedge \theta \wedge \alpha . 
\end{eqnarray*}

The proofs of b) and c) are similar.
\end{proof}

\begin{proof} (Of Lemma \ref{LemmaIntersectionOfVarieties})
Let $\theta$ be a smooth $(p,p)$ form having the same cohomology class as that of $[V_1]$. Then by Proposition 2.1 in \cite{dinh-sibony5}, there are positive $(p-1,p-1)$ currents $R^{\pm}$ so that $[V_1]-\theta =dd^c(R^+-R^-)$. Moreover, $R^{\pm}$ are $DSH$ and we can choose so that $R^{\pm}$ are continuous outside $V_1$. To prove Lemma \ref{LemmaIntersectionOfVarieties}, it suffices to show that there is a current $S$ so that for any good approximation scheme by $C^2$ forms $\mathcal{K}_n$ then 
\begin{eqnarray*}
\lim _{n\rightarrow\infty}\mathcal{K}_n(R^+-R^-)\wedge [V_2]=S.
\end{eqnarray*}
The sequence $\mathcal{K}_n^{\pm}(R^{\pm})\wedge [V_2]$ converges on $Y-V_1\cap V_2$. In fact, outside of $V_2$ then $\mathcal{K}_n^{\pm}(R^{\pm})\wedge [V_2]=0$, and outside of $V_1$ then $\mathcal{K}_n^{\pm}(R^{\pm})$ converges locally uniformly (by condition 4) of Definition \ref{DefinitionGoodApproximation}) to a continuous form and hence $\mathcal{K}_n^{\pm}(R^{\pm})\wedge [V_2]$ converges. Then by an argument as in the proof of Theorem 6 in \cite{truong} using the Federer-type support theorem in Bassanelli \cite{bassanelli}, we are done.   
\end{proof}


\end{document}